\documentclass[12pt]{amsart}
\usepackage{a4wide}
\usepackage{amssymb}
\usepackage{amsthm}
\usepackage{amsmath}
\usepackage{amscd}
\usepackage{verbatim}
\usepackage{mathrsfs}
\usepackage{dsfont,bbm} 
\usepackage{youngtab}
\usepackage[usenames]{color}
\usepackage{tikz,graphicx}
\usepackage{hyperref}

\numberwithin{equation}{section}
\theoremstyle{plain}
\newtheorem{theorem}{Theorem}[section]

\newtheorem{proposition}[theorem]{Proposition}

\theoremstyle{definition}

\newtheorem*{example}{Example}
\theoremstyle{remark}
\newtheorem*{remark}{Remark}

\newcommand{\Rat}{\mathbb Q}

\newcommand{\Z}{\mathbb Z}

\newcommand{\abs}[1]{\lvert#1\rvert}
\newcommand{\la}{\lambda}

\newcommand{\boop}{\vspace{0.5cm}}

\begin{document}

\title[Combinatorial Properties of Rogers-Ramanujan-Type Series]
{Combinatorial Properties of Rogers-Ramanujan-Type Identities Arising from Hall-Littlewood Polynomials}

\author{Claire Frechette and Madeline Locus}

\address{Brown University, 
69 Brown St, BOX 3390, Providence, RI 02912}
\email{claire\_frechette@brown.edu}

\address{2710 Sugarwood Dr, Sugar Land, TX 77478}
\email{maddie93@uga.edu}

\thanks{We would like to thank the NSF for supporting the Emory REU in Number Theory}

\subjclass[2010]{11P84, 05A19}

\keywords{Rogers-Ramanujan identities, Hall-Littlewood polynomials, partition congruences}

\begin{abstract} 
Here we consider the $q$-series coming from the Hall-Littlewood polynomials,
\begin{equation*}
R_\nu(a,b;q)=\sum_{\substack{\la \\[1pt] \la_1\leq a}}
q^{c\abs{\la}} P_{2\la}\big(1,q,q^2,\dots;q^{2b+d}\big).
\end{equation*}
These series were defined by Griffin, Ono, and Warnaar in their work on the framework of the Rogers-Ramanujan identities. We devise a recursive method for computing the coefficients of these series when they arise within the Rogers-Ramanujan framework. Furthermore, we study the congruence properties of certain quotients and products of these series, generalizing the famous Ramanujan congruence
\begin{equation*}
p(5n+4)\equiv0\pmod{5}.
\end{equation*}
\end{abstract}

\maketitle

\section{Introduction and Statement of Results}\label{Intro}
%recall RR identities
The famous Rogers-Ramanujan identities
\begin{align}\label{G}
G(q)&:=\sum_{n=0}^\infty\frac{q^{n^2}}{(1-q)\cdots(1-q^n)}
=\prod_{n=0}^\infty\frac{1}{(1-q^{5n+1})(1-q^{5n+4})}\\
\label{H}
H(q)&:=\sum_{n=0}^\infty\frac{q^{n^2+n}}{(1-q)\cdots(1-q^n)}
=\prod_{n=0}^\infty\frac{1}{(1-q^{5n+2})(1-q^{5n+3})}
\end{align}
%explain Framework series sides
\cite[p. 384]{HW} have inspired discoveries in many fields of mathematics and physics, such as algebraic geometry, group theory, knot theory, modular forms, orthogonal polynomials, statistical mechanics, probability, and transcendental number theory (see references in \cite{GOW}). Much of their importance stems from the fact that the Rogers-Ramanujan infinite sums correspond to $q$-series which happen to equal infinite products. Recently, Griffin, Ono, and Warnaar unearthed a more general framework for Rogers-Ramanujan-type identities where the infinite sum of symmetric functions is equal to an infinite product with periodic exponents, which turns out essentially to be a modular function. These series are defined using Hall-Littlewood polynomials $P_\lambda(x;q)$ (see Section 3.1 for definitions) \cite{GOW}. 

To define these polynomials, we require notation for an \textit{integer partition}, a non-increasing sequence of nonnegative integers with finitely many nonzero terms. For a partition $\lambda = (\lambda_1,\lambda_2,...)$, we let $|\lambda| := \lambda_1 + \lambda_2+ ...$ and $2\lambda := (2\lambda_1,2\lambda_2,...)$. In particular, if $\lambda$ is a partition of length $n$, then the corresponding Hall-Littlewood polynomial is a symmetric function in $n$ variables. So, for particular ordered pairs $\nu = (c,d)$ and arbitrary $a,b\geq 1$, we consider $q$-series of the form

\begin{equation}\label{genR}
R_\nu(a,b;q) = \sum_{\substack{\la \\[1pt] \la_1\leq a}}
q^{c\abs{\la}} P_{2\la}\big(1,q,q^2,\dots;q^{2b+d}\big) 
\end{equation}
%example revisit RR ids
which correspond the left hand sides of the Rogers-Ramanujan identities if we take $a = b = 1$ and $d = -1$. Then, the series becomes
\[\sum_{\substack{\la \\[1pt] \la_1\leq 1}}
q^{c\abs{\la}} P_{2\la}\big(1,q,q^2,\dots;q\big)= \sum_{n=0}^\infty
q^{c|(1^n)|} P_{(2^n)}\big(1,q,q^2,\dots;q\big)
\]
since the sum is over the empty partition and partitions containing $n$ copies of 1, which we denote as $\la = (1^n)$, for all $n\geq 1$. Since 
\[ 
q^{c|(1^n)|} P_{(2^n)}\big(1,q,q^2,\dots;q\big) = \frac{q^{n^2 + (c-1)n}}{(1-q) \cdots (1-q^n)}
\]
this expression yields the series side of the Rogers-Ramanujan identities \eqref{G} and \eqref{H} when $c = 1,2$.
%natural question: use structure ids to recursively det coeffs, so can compute series, we do that here
The work of Griffin, Ono, and Warnaar shows that for special choices of $\nu$, $R_\nu(a,b,q)$ equals an infinite product of modular functions \cite{GOW}. Notice that this framework then gives a way to define the $q$-series sides of the Rogers-Ramanujan type identities in terms of symmetric functions. 

Therefore, the main objects of study here are the $q$-series defined by\eqref{genR}. First, we show that this framework gives rise to a recursive method of computing the coefficients of these $q$-series, which allows us to directly compute these $q$-series.
%define notation for c_*
\cite{GOW} proves that the expression in \eqref{genR} equals an infinite product of terms $(1-q^n)$. For the specific pairs $\nu$ in $\{(1,-1),(2,-1),(1,0),(2,-2)\}$, the pairs considered in \cite{GOW}, we define $c_\nu(a,b;t)$ as follows:
\begin{equation}
R_\nu(a,b;q) = \prod _{t=1}^\infty (1-q^t)^{c_\nu(a,b;t)}.
\end{equation}
The exponents $c_\nu(a,b;t)$ are easy to compute using Theorems 1.1, 1.2, and 1.3 from \cite{GOW} (see Section \ref{frame}). In particular, $c_\nu(a,b;t)$ show striking  patterns and uniformity. For example, if $a$ or $b$ is 1, we have the following descriptions, where $\kappa = 2a + 2b+ 1$,
\begin{equation}
\begin{split}
 c_{(1,-1)}(1,b;t) &= \begin{cases} 0 & \text{if $t \equiv 0, \pm 2 \pmod{\kappa}$}\\ -1& \text{else}\end{cases}\\
 c_{(1,-1)}(a,1;t) &= \begin{cases} 0 & \text{if $t \equiv 0, \pm (a+1) \pmod{\kappa}$}\\ -1& \text{else}\end{cases}\\
 c_{(2,-1)}(1,b;t) &= \begin{cases} 0 & \text{if $t \equiv 0, \pm 1 \pmod{\kappa}$}\\ -1 & \text{else}\end{cases}\\
 c_{(2,-1)}(a,1;t) &= \begin{cases} 0 & \text{if $t \equiv 0, \pm 1 \pmod{\kappa}$}\\ -1 & \text{else}\end{cases}
\end{split}
\end{equation}
Using these exponents, we find that the coefficients of $R_\nu$ are intertwined with a particularly notable set of universal polynomials independent of $a,b$: if $n \geq 2$, then we define
%define set of F_n universal polynomials
\begin{equation}\label{defF}
\begin{split}
 \widehat{F}_n(x_1,&...,x_{n-1}) \\
 &:= \sum \limits_{\substack{m_1,...,m_{n-1} \geq 0 \\ m_1 + 2m_2 + ... + (n-1)m_{n-1} = n}} (-1)^{m_1 + ... + m_{n-1}} \cdot \frac{(m_1 + ... + m_{n-1} - 1)!}{m_1!\cdots m_{n-1}!} \cdot x_1^{m_1} \cdots x_{n-1}^{m_{n-1}}.
\end{split}
\end{equation}
For example, the first few polynomials $\widehat{F}_n$ are
\begin{align*}
\widehat{F}_2 &= \frac{1}{2}x_1^2\\
\widehat{F}_3 &= -\frac{1}{3}x_1^3 + x_1x_2\\
\widehat{F}_4 &= \frac{1}{4}x_1^4 + x_1x_3 - x_1^2x_2 + \frac{1}{2}x_2^2
\end{align*}
These polynomials allow us to recursively determine the coefficients of $R_\nu(a,b;q)$. 
%THM: assuming notation above, we have a_* = F_n + sum
\begin{theorem}\label{coeffs}
Assuming the notation above, if $\nu$ is in the set $\{(1,-1),(2,-1),(1,0),(2,-2)\}$, suppose 
\[
R_\nu(a,b;q) = 1 + \sum_{n = 1}^\infty a_\nu(n)q^n.
\]
Then, 
\begin{equation*}
a_\nu(n) = \widehat{F}_{n}(a_\nu(1),...,a_\nu(n-1)) - \frac{1}{n} \sum_{d|n} c_\nu(a,b;d)\cdot d.
\end{equation*}
\end{theorem}

\begin{example}
For instance, if we consider the coefficients $a_{(1,-1)}$ of $R_{(1,-1)}(1,1;q)$, which we recall gives us \eqref{G},
\begin{equation*}
 R_{(1,-1)}(1,1;q)=\sum_{n=0}^\infty\frac{q^{n^2}}{(1-q)\cdots(1-q^n)}
=\prod_{t=0}^\infty\frac{1}{(1-q^{5t+1})(1-q^{5t+4})}
\end{equation*}
we find that our values for $c_{(1,-1)}(1,1;t)$ are
\begin{center}
\begin{tabular}{|c c|c|c|c|c|}
\hline$t\pmod{5}:$&0&1&2&3&4\\
\hline $c_{(1,-1)}(1,1;t):$&0&-1&0&0&-1\\
\hline
\end{tabular}.
\end{center}
Thus, the first few terms of the corresponding $q$-series are
\[
R_{(1,-1)}(1,1;q) = 1 + q + q^2 + q^3 + 2q^4 + 2q^5 + 3q^6 + 3q^7 + 4q^8 + 5q^9 + 6q^{10} + ....
\]
We can then calculate $a_{(1,-1)}(6)$ as follows:
\begin{equation*}
a_{(1,-1)}(6) = \widehat{F}_{6}(a_{(1,-1)}(1),...,a_{(1,-1)}(5)) - \frac{1}{6}\sum_{d|6}c_{(1,-1)}(1,1;d)\cdot d.
\end{equation*}
By the definition of $\widehat{F}_n$, we have that
\begin{equation*}
\begin{split}
\widehat{F}_{6}(x_1,...,x_5) = \frac{1}{6}x_1^6 - x_1^4x_2 + &\frac{3}{2}x_1^2x_2^2 + x_1^3x_3 - \frac{1}{3}x_2^3\\& - 2x_1x_2x_3 -
x_1^2x_4 + \frac{1}{2}x_3^2 + x_2x_4 + x_1x_5.
\end{split}
\end{equation*}
Substituting in our values for $x_1,...,x_5$, we find
\begin{equation*}
\widehat{F}_6(a_{(1,-1)}(1),...,a_{(1,-1)}(5)) = \widehat{F}_6(1,1,1,2,2) = 2 - \frac{1}{6} .
\end{equation*}
So, noting that $1,2,3$ and $6$ are the only divisors of $6$ allows us to then conclude that
\begin{equation*}
a_{(1,-1)}(6) = \frac{11}{6} + \frac{1}{6} + 1= 3.
\end{equation*}
\end{example}
\boop
The Rogers-Ramanujan identities are known to be of number-theoretic interest, standing out not only for their partition interpretations, but also for their combinatorial properties. Perhaps the most famous example of these properties is the fact that the quotient of $H$ and $G$ is the famous Rogers-Ramanujan continued fraction
\begin{equation}\label{H/G}
\frac{H(q)}{G(q)} = \cfrac{1}{1+\cfrac{q}{1+\cfrac{q^2}{1+\cfrac{q^3}
{\,\ddots}}}}.
\end{equation}
A more obvious and perhaps overlooked fact is that their product has interesting properties as well.
\begin{align}\label{GH}
G(q)\cdot H(q) &= \prod_{t=0}^\infty \frac{1}{(1-q^{5n+1})(1-q^{5n+4})}\cdot \frac{1}{(1-q^{5n+2})(1-q^{5n+3})}\\
\notag &= \prod_{t=1}^\infty \frac{(1-q^{5t})}{(1-q^t)} = \sum_{n=0}^\infty d_5(n)q^n.
\end{align}
It is a fact that the coefficients $d_5(n)$ of the product $G(q)\cdot H(q)$ for $n \equiv 4 \pmod{5}$ are multiples of 5. This is equivalent to the famous Ramanujan congruence for the partition function \cite{HW}
\begin{equation}\label{p54}
p(5n+4) \equiv 0 \pmod{5}
\end{equation}
which follows from elementary $q$-series identities of Euler and Jacobi and the fact that
\begin{equation*}
\prod_{t=1}^\infty \frac{(1-q^{5n})}{(1-q^n)} \equiv \prod_{t=1}^\infty (1-q^n)^4 \pmod{5}.
\end{equation*}

In view of \eqref{H/G} and \eqref{GH}, it is natural to ask if products of $R_\nu(a,b;q)$ for different $\nu,a,b$ form similar patterns. Searching for products of the form
\begin{equation}
\Psi_m(q):=\prod_{t=1}^\infty \frac{(1-q^{mt})}{(1-q^t)} =1 + \sum_{n=1}^\infty d_m(n)q^n
\end{equation}
we generalize the concept developed by Rogers and Ramanujan, considering ratios of the products of these functions. 
\begin{remark}
Note that $\Psi_m$ can also be viewed as the generating function counting partitions which are \emph{$m$-regular}, meaning partitions whose parts which are not multiples of $m$. Thus, we can view the coefficients $d_m(n)$ as counting $m$-regular partitions of $n$.
\end{remark}
When $m=9$, we have the well-known example discovered by Freeman Dyson, as mentioned in \cite{GOW},
\begin{equation*}
R_{(1,-1)}(2,2;q) = \prod_{t=1}^\infty \frac{(1-q^{9t})}{(1-q^t)}.
\end{equation*}
In fact, there exist such products for many other choices of $m$.

For all but finitely many $m$, there are infinitely many such representations. In particular, for even $m\geq8$, and $m\equiv 1\pmod{4}$, we have the following explicit formulae:
\vspace{0.3 cm}
\begin{theorem}\label{evens} Assuming the notation above, the following are true:\\
$(1)$ If $m \geq 8$ and $m$ is even, then we have
\begin{equation*}
\frac{R_{(1,0)}(2, \frac{m}{2}-3;q)\cdot R_{(2,-2)}(\frac{m}{2}-2,2;q)}{R_{(2,-2)}(\frac{m}{2}-3,3;q)} = \prod_{t=1}^\infty \frac{(1-q^{mt})}{(1-q^t)}.
\end{equation*}
$(2)$ If $m\equiv 1 \pmod{4}$, $m>1$, then we have
\begin{equation*}
\frac{R_{(1,-1)}(\frac{m-1}{2}-1, 1;q)\cdot R_{(2,-1)}(\frac{m-1}{4},\frac{m-1}{4};q)}{R_{(2,-1)}(\frac{m-1}{4}+1,\frac{m-1}{4}-1;q)} = \prod_{t=1}^\infty \frac{(1-q^{mt})}{(1-q^t)}.
\end{equation*}
\end{theorem}
\vspace{0.3cm}
%SHOULD THIS GO IN?
%For m equiv 3 mod 4, the reader can easily verify that there exist such representations.
\begin{remark}
For $m\equiv3 \pmod{4}$, there do not seem to be any similar expressions exactly giving $\Psi_m$; however, there are infinitely many such representations giving $\Psi_m^2$, although we have been unable to find a simple closed formula for all such $m$. In fact, for general $m$, there are infinitely many representations, due to the existence of nontrivial kernels. In particular, we have the formula, for $m\equiv 3 \pmod{4}$,
\begin{equation*}
\frac{R_{(1,-1)}(\frac{m-3}{2} - 1,2) \cdot R_{(2,-1)}(\frac{m-3}{4},\frac{m-3}{4}+1)}{R_{(1,-1)}(1,\frac{m-3}{2}) \cdot R_{(2,-1)}(1,\frac{m-3}{2}) \cdot R_{(2,-1)}(\frac{m-3}{4}-1,\frac{m-3}{4}+2)}  = 1,
\end{equation*}
which is just one of the many nontrivial elements of the kernel.
\end{remark}
In view of \eqref{p54}, it is natural to ask whether the series in Theorem \ref{evens} possess similar congruences of the form $d_m(an+b)\equiv0\pmod{m}$.
\boop
\begin{theorem}\label{aps}
If $m=p^t$, $p$ prime, then there are infinitely many non-nested arithmetic progressions $an+b$ such that $d_m(an+b)\equiv0\pmod{p}$.
\end{theorem}
\begin{example}Here, we list a few specific congruences:
\begin{align*}
 d_5(121n+9) &\equiv 0 \pmod{5}\\
d_7(9n+5) &\equiv 0 \pmod{7}\\
d_9(4n +3) &\equiv 0 \pmod{3}\\
 d_{11}(49n+6) &\equiv 0 \pmod{11}\\
 d_{17}(157n+104) &\equiv 0 \pmod{17}\\
 d_{19}(243n+141) &\equiv 0 \pmod{19}\\
 d_{31}(463n+346) &\equiv 0 \pmod{31}.
\end{align*}
\end{example}
These are proven by finite calculation. By a theorem of Sturm, it is possible to check whether modular forms are identically $0\pmod{m}$ by checking only a finite number of coefficients \cite[p. 40]{CBMS}. Applying this theorem to modular forms yields these congruences, which we will consider further in Section \ref{proof3}.

The rest of the paper is organized as follows. The proof of Theorem \ref{coeffs} follows from understanding the logarithmic derivatives of the Rogers-Ramanujan series arising from the Hall-Littlewood polynomials. To obtain Theorem \ref{coeffs} from these results, we discuss universal recursion relations for infinite products in Section \ref{symm}. In Section \ref{pf12}, we recall the combinatorial structure of $q$-series in the work of Griffin, Ono, and Warnaar, and we use the results of Section \ref{symm} to prove Theorem \ref{coeffs}. Theorem \ref{evens}, giving closed formulas for the infinite product $\Psi_m$, is also obtained in Section \ref{pf12}. The theory of the congruences relies on the theory of modular forms; see \cite{CBMS} for examples. In Section \ref{proof3}, we apply well-known facts about the theory of modular forms modulo $p$ and Hecke operators to prove Theorem \ref{aps}.
\boop

\section{Symmetric Functions and Infinite Products}\label{symm}
%"THM 3 and proof" note we want to use t for n
In order to prove Theorem \ref{coeffs}, we must first determine a relationship between the exponents of an infinite product and the coefficients of the equivalent $q$-series.
\begin{theorem} \label{gencoeff}If $f = 1 + \sum \limits_{n=1}^\infty a(n)q^n = \prod \limits_{t=1}^\infty (1-q^t)^{c(t)}$, then for every integer $n \geq 1$ we have
\begin{equation}\label{an}
a(n) = \widehat{F}_n(a(1),...,a(n-1)) - \frac{1}{n}\sum \limits_{d|n}c(d)\cdot d.
\end{equation}
\end{theorem}

\begin{proof}
For convenience, define 
\begin{equation}
b(n):= \sum_{d|n} c(d)d.
\end{equation}
We proceed to search for a closed formula for $b(n)$ in terms of $a(n)$. To this end, we set the two equivalent expressions for $f$ as equal and take the logarithmic derivatives of both sides. Taking $\log$ of both sides, we have
\begin{equation}\label{log}
\log(1+ \sum_{n=1}^\infty a(n)q^n) = \sum_{n=1}^\infty c(n) \log(1-q^n).
\end{equation}
We then recall the series expansion, which we use to expand the right hand side,
\[
\log(1-x) = -\sum_{n=1}^\infty \frac{x^n}{n}.
\]
Then, taking the derivative of each side and reindexing the right hand side, we have
\begin{align*}
\frac{\sum_{n=1}^\infty n \cdot a(n)q^{n-1}}{1 + \sum_{n= 1}^\infty a(n)q^n} &= -\sum_{n=1}^\infty\sum_{d|n} c(d)d q^{n-1}\\
&=  - \sum_{n=1}^\infty b(n)q^{n-1}.
\end{align*}
From this relationship, we multiply both sides by the denominator of the left hand side, which gives us
\begin{equation*}
\sum_{n=1}^\infty n \cdot a(n)q^n = (- \sum_{n=1}^\infty b(n)q^n) \cdot (1 + \sum_{n=1}^\infty a(n)q^n).
\end{equation*}
Expanding the right hand side and equating coefficients of $q^{n}$, we then observe the identity
\begin{equation}
0 = b(n) + b(n-1)a(1) + b(n-2)a(2) + ... + b(1)a(n-1) + na(n).
\end{equation}
We can also see a similar relationship when we consider the symmetric power functions $s_i$ and the elementary symmetric functions $\sigma_i$, \begin{equation}
s_i := X_1^i + X_2^i + ... + X_n^i \hspace{1.5 cm} \sigma_i = \sum \limits_{1 \leq j_1 < j_2 < ... < j_k \leq n} X_{j_1}...X_{j_k}
\end{equation}
which easily combine to give us the identity
\begin{equation}\label{sigmas}
0 = s_n - s_{n-1}\sigma_1 + s_{n-2}\sigma_2 - ... + (-1)^{n-1}s_1\sigma_{n-1} + (-1)^n\sigma_n.
\end{equation}

When we evaluate \eqref{sigmas} at $(X_1,...,X_n) = (\lambda(1,n),...,\lambda(n,n))$, where $\lambda(j,n)$ are the roots of the polynomial
\begin{equation*}
X^n + a(1) X^{n-1} + a(2)X^{n-2} + ... + a(n),
\end{equation*}
we have that $a(n) = \sigma_n$ for $n \geq 1$. Then, we must have a relationship between the $b(n)$ and the $s_n$ in order for the first identity to be true. Matching up the $a(i)$ terms in each equation, we find 
\begin{equation}
b(n) = (-1)^ns_n.
\end{equation}
We then use the fact that
\begin{equation}
s_n = n \sum \limits_{\substack{m_1,...,m_n \geq 0 \\ m_1 + 2m_2 + ... + nm_n = i}} (-1)^{m_2 + m_4 + ...} \cdot \frac{(m_1 + m_2 + ... + m_n - 1)!}{m_1!m_2! \cdots m_n!} \cdot \sigma_1^{m_1} \cdots \sigma_n^{m_n}.
\end{equation}
So, noting that $n = \sum jm_j$, we look at the exponent of the $(-1)$ in $b(n)$, and recall that exponents on $(-1)$ are calculated modulo 2, so we have that
\begin{align*}
n + (m_2 + m_4 + ...) &\equiv (m_1 + 2m_2 + ... + nm_n) + (m_2 + m_4 + ...) \\
&\equiv m_1 + m_3 + ... + m_2 + m_4 + ...\\
&\equiv m_1 + m_2 + m_3  + ... + m_n.
\end{align*}
Then, we have
\begin{equation}
b(n) = n \sum \limits_{\substack{m_1,...,m_n \geq 0 \\ m_1 + 2m_2 + ... + nm_n = n}} (-1)^{m_1 + m_2 + ... + m_n} \cdot \frac{(m_1 + m_2 + ... + m_n - 1)!}{m_1!m_2! \cdots m_n!} \cdot \sigma_1^{m_1} \cdots \sigma_n^{m_n}.
\end{equation}
In particular, note that the $m_j$ can be viewed as the multiplicities of parts corresponding to a partition of $n$, as we will see in the next section, so this sum in $b(n)$ sums over all possible partitions of $n$. We then separate one partition from the sum in $b(n)$: the partition $(n)$, where there is one part of size $n$. It is easy to see that following the formula for $b(n)$, the associated summand for this partition will be $-(n)a(n+1)$. So, we define $\widehat{F}_n(x_1,...,x_{n-1})$ as in \eqref{defF}, and note that
\begin{equation}
\widehat{F}_n(a(1),...,a(n-1)) - a(n) = \frac{1}{n}b(n).
\end{equation}
So we have that
\begin{equation}
a(n) = \widehat{F}_n(a(1),...,a(n-1)) - \frac{1}{n} \sum \limits_{d|n} c(d)\cdot d.
\end{equation}
\end{proof}

\section{Proof of Theorems 1 and 2}\label{pf12}

\subsection{Hall-Littlewood Polynomials and Hall-Littlewood $q$-series}

Let $\lambda := (\lambda_1,\lambda_2,...)$ be an integer partition, a nonincreasing sequence of nonnegative integers, with finitely many nonzero terms.
As we defined earlier, the length of $\lambda$, $l(\lambda)$, is the number of parts in $\lambda$, and the size of $\lambda$, $|\lambda|$, is the sum of the parts  in $\lambda$. We also define $m_i := m_i(\lambda)$ as the multiplicities of parts of $\lambda$ with size $i$. So, $|\lambda| = \sum_i im_i$.

In order to consider the Hall-Littlewood polynomial of a particular partition $\lambda$, we fix a positive integer $n$ such that $n \geq l(\lambda)$. We then define $x := (x_1,...,x_n)$ and $x^\lambda := x_1^{\lambda_1}\cdots x_n^{\lambda_n}$, where $\lambda_i = 0$ for $i > l(\lambda)$.

We also define $v_\lambda(q) := \prod \limits_{i=0}^n \frac{(q)_{m_i}}{(1-q)^{m_i}}$, where $m_0 := n - l(\lambda)$. Note that here, the $(q)_{m_i}$ uses Pochhammer notation, defined as
\begin{equation*}
(a)_k := (a;q)_k = \begin{cases} (1-a)(1-aq)\cdots (1-aq^{k-1}) &: k \geq 0\\ \prod \limits_0^\infty (1-aq^n) &: k =\infty \end{cases}.
\end{equation*}

The \textit{Hall-Littlewood polynomial} $P_\lambda(x;q)$ is then defined as the symmetric function
\begin{equation}
P_\lambda(x;q) := \frac{1}{v_\lambda(q)} \sum \limits_{w \in S_n} w\left(x^\lambda \cdot \prod \limits_{i<j}\frac{x_i - qx_j}{x_i-x_j}\right).
\end{equation}
where the symmetric group $S_n$ acts on $x$ by permuting the $x_i$. Since $P_\lambda(x;q)$ is a symmetric function and a homogeneous polynomial, meaning that all terms have the same degree $|\lambda|$, we can represent it in terms of the $r$-th power sum symmetric functions
\begin{equation*}
p_r := p_r(x) = x_1^r + x_2^r + ...
\end{equation*}
which form a $\Rat$-basis of the ring of symmetric functions in $n$ variables. So, we can write
\begin{equation*}
P_\lambda(x;q) = \sum a_jp_j.
\end{equation*}
In particular, when $x = (1,q,q^2,...)$, $p_j$ becomes a geometric series in $q$, so there exists an endomorphism sending $p_j \mapsto \frac{1}{1-q^j}$.

We then use the ring homomorphism 
\begin{equation}
\phi_q(p_r):=\frac{p_r}{(1-q^r)}
\end{equation}
to calculate Hall-Littlewood polynomials by transforming $P_\lambda$ into a truncated version of itself. We define the \textit{modified Hall-Littlewood polynomials} $P_\lambda^\prime(x;q)$ as the image of the set of $P_\lambda(x;q)$ under $\phi_q$, so
\[
P^\prime_\lambda := \phi_q(P_\lambda).
\]
We then have that this modified polynomial allows us to evaluate $P_\lambda$ at an infinite geometric progression, since
\begin{equation*}
P_\lambda(1,q,q^2,...;q^n) = P^\prime_\lambda(1,q,...,q^{n-1};q^n)
\end{equation*}
which we can see by
\begin{equation*}
p_r(1,q,q^2,...) = \frac{1}{1-q^r} = \frac{1-q^{nr}}{1-q^r}\frac{1}{1-q^{nr}} = \phi_{q^n}(p_r(1,q,...,q^{n-1})).
\end{equation*}

\subsection{General Framework Theorems}\label{frame}\emph{}\\
For convenience in expressing $R_\nu(a,b;q)$ as a product, we first define a modified theta function
\[
\theta(a;q):=(a;q)_{\infty}(q/a;q)_{\infty}.
\]
Note that this function uses the Pochhammer notation defined earlier.
\begin{theorem}[Griffin,Ono,Warnaar]\label{Thm_Main}
If $a$ and $b$ are positive integers and $\kappa:=2a+2b+1$, then
we have that
\begin{align}%\label{RRfirst}
\notag R_{(1,1)}(a,b;q) := \sum_{\substack{\la \\[1pt] \la_1\leq a}}
q^{\abs{\la}} P_{2\la}\big(1,q,&q^2,\dots;q^{2b-1}\big) \\
&=\frac{(q^{\kappa};q^{\kappa})_{\infty}^b}{(q)_{\infty}^b}\cdot 
\prod_{i=1}^b  \theta\big(q^{i+a};q^{\kappa}\big)
\prod_{1\leq i<j\leq b} 
\theta\big(q^{j-i},q^{i+j-1};q^{\kappa}\big) \notag \\
&=\frac{(q^{\kappa};q^{\kappa})_{\infty}^a}{(q)_{\infty}^a}\cdot 
\prod_{i=1}^a  \theta\big(q^{i+1};q^{\kappa}\big)
\prod_{1\leq i<j\leq a} 
\theta\big(q^{j-i},q^{i+j+1};q^{\kappa}\big), \notag
\intertext{and}
%\label{RRsecond}
\notag R_{(1,2)}(a,b;q):=\sum_{\substack{\la \\[1pt] \la_1\leq a}}
q^{2\abs{\la}} P_{2\la}\big(1,q,&q^2,\dots;q^{2b-1}\big) \\
&=\frac{(q^{\kappa};q^{\kappa})_{\infty}^b}{(q)_{\infty}^b}\cdot 
\prod_{i=1}^b  \theta\big(q^i;q^{\kappa}\big)
\prod_{1\leq i<j\leq b} 
\theta\big(q^{j-i},q^{i+j};q^{\kappa}\big) \notag \\
&=\frac{(q^{\kappa};q^{\kappa})_{\infty}^a}{(q)_{\infty}^a}\cdot 
\prod_{i=1}^a  \theta\big(q^i;q^{\kappa}\big)
\prod_{1\leq i<j\leq a} 
\theta\big(q^{j-i},q^{i+j};q^{\kappa}\big). \notag
\end{align}
\end{theorem}

We also have the even modulus analogs of the two formulas in Theorem \ref{Thm_Main}, respectively.
\begin{theorem}[Griffin, Ono,Warnaar]\label{Thm_Main2}
If $a$ and $b$ are positive integers and $\kappa:=2a+2b+2$, then we have that
\begin{align}\label{RRthird}
\notag R_{(1,0)}(a,b;q):=&\sum_{\substack{\la \\[1pt] \la_1\leq a}}
q^{\abs{\la}} P_{2\la}\big(1,q,q^2,\dots;q^{2b}\big) \\
&=\frac{(q^2;q^2)_{\infty}(q^{\kappa/2};q^{\kappa/2})_{\infty}
(q^{\kappa};q^{\kappa})_{\infty}^{b-1}}{(q)_{\infty}^{b+1}}  
\cdot 
\prod_{i=1}^b  \theta\big(q^i;q^{\kappa/2}\big)
\prod_{1\leq i<j\leq b} \theta\big(q^{j-i},q^{i+j};q^{\kappa}\big)
\notag \\
&=\frac{(q^{\kappa};q^{\kappa})_{\infty}^a}{(q)_{\infty}^a} 
\cdot \prod_{i=1}^a  \theta\big(q^{i+1};q^{\kappa}\big)
\prod_{1\leq i<j\leq a} 
\theta\big(q^{j-i},q^{i+j+1};q^{\kappa}\big). \notag
\end{align}
\end{theorem}

\begin{theorem}[Griffin, Ono, Warnaar]\label{Thm_Main3}
If $a$ and $b$ are positive integers such that $b\geq 2$, and $\kappa:=2a+2b$, then we have that
\begin{align}%\label{RRfourth}
\notag R_{(2,-2)}(a,b;q):=\sum_{\substack{\la \\[1pt] \la_1\leq a}}
q^{2\abs{\la}} P_{2\la}\big(1,&q,q^2,\dots;q^{2b-2}\big) \\
&=\frac{(q^{\kappa};q^{\kappa})_{\infty}^b}
{(q^2;q^2)_{\infty}(q)_{\infty}^{b-1}} 
\cdot \prod_{1\leq i<j\leq b} \theta\big(q^{j-i},q^{i+j-1};q^{\kappa}\big) \notag \\
&=\frac{(q^{\kappa};q^{\kappa})_{\infty}^a}{(q)_{\infty}^a} 
\cdot \prod_{i=1}^a  \theta\big(q^i;q^{\kappa}\big)
\prod_{1\leq i<j\leq a} 
\theta\big(q^{j-i},q^{i+j};q^{\kappa}\big). \notag
\end{align}
\end{theorem}

\subsection{Proof of Theorem \ref{coeffs}} 
From the above theorems, we know that we can write $R_\nu(a,b;q)$ as a product of terms $(1-q^t)^{c_\nu(a,b;t)}$. We then apply Theorem 2.1 to determine the coefficients of $R_\nu(a,b;q)$.
\begin{proof} [Proof of Theorem \ref{coeffs}] For each of the identities in Theorems \ref{Thm_Main}, \ref{Thm_Main2}, and \ref{Thm_Main3}, we have a distinct formula
\begin{equation*}
R_\nu(a,b;q) = 1+ \sum_{n=1}^\infty a_\nu(n)q^n = \prod_{t=1}^\infty (1-q^t)^{c_\nu(a,b;t)}. 
\end{equation*}
Thus, applying Theorem \ref{gencoeff}, we determine that
\begin{equation}
a_\nu(n) = \widehat{F}_n((a(1),...,a(n-1)) - \frac{1}{n} \sum \limits_{d|n} c_\nu(a,b;d)\cdot d.
\end{equation}
\end{proof}

\subsection{Proof of Theorem \ref{evens}}\label{proof2}
The proof of Theorem \ref{evens} relies on the set of three framework theorems introduced in Section 3.2. Note that in the framework theorems, there are two possible expressions for the product sides; we will choose either expression, based on convenience.
\begin{proof} [Proof of Theorem \ref{evens} (1)]
By Theorem 3.3, we have that for $m = 2a + 2b$
\begin{equation*}
R_{(2,-2)}(a,b;q) = \frac{(q^m;q^m)^a_\infty}{(q)^a_\infty}\cdot\prod_{i=1}^a \theta(q^i;q^m)\prod_{1\leq i<j\leq a} \theta(q^{j-i},q^{i+j};q^m).
\end{equation*}
Considering then the fraction below, we find that the denominator cancels out, and we are left with
\begin{equation*}
\frac{R_{(2,-2)}(\frac{m}{2}-2,2;q)}{R_{(2,-2)}(\frac{m}{2}-3,3;q)}=
\frac{(q^m;q^m)_\infty}{(q)_\infty} \theta(q^{\frac{m}{2}-2};q^m) \prod_{i=1}^{\frac{m}{2}-3} \theta(q^{\frac{m}{2}-2+i};q^m)\theta(q^{\frac{m}{2}-2-i};q^m).
\end{equation*} 
Splitting up the theta notation and rearranging the indices, we find that the right hand side of this equation can be rewritten as
\begin{align*}
 \frac{(q^m;q^m)_\infty}{(q)_\infty} &\cdot\prod_{i=1}^{4} \theta(q^i;q^m)\prod_{i=5}^{\frac{m}{2}}\theta(q^i;q^m)^2\\
&= \prod_{i=5}^{\frac{m}{2}}\theta(q^i;q^m).
\end{align*} 
Then, from Theorem 3.2, we have that for $m = 2a + 2b + 2$,
\begin{equation*}
R_{(1,0)}(a,b;q) = \frac{(q^m;q^m)^a_\infty}{(q)^a_\infty}\cdot\prod_{i=1}^a \theta(q^{i+1};q^m)\prod_{1\leq i<j\leq a} \theta(q^{j-i},q^{i+j+1};q^m).
\end{equation*}
So then we have that
\begin{equation*}
R_{(1,0)}\left(2,\frac{m}{2}-3;q\right) = \frac{(q^m;q^m)^2_\infty}{(q)^2_\infty}\prod_{i=1}^4 \theta(q^i;q^m).
\end{equation*}
Therefore, 
\begin{align*}
\frac{R_{(1,0)}(2, \frac{m}{2}-3;q)\cdot R_{(2,-2)}(\frac{m}{2}-2,2;q)}{R_{(2,-2)}(\frac{m}{2}-3,3;q)} &= \prod_{i=5}^{\frac{m}{2}}\theta(q^i;q^m)\cdot \frac{(q^m;q^m)^2_\infty}{(q)^2_\infty}\prod_{i=1}^4 \theta(q^i;q^m)\\
&= \frac{(q^m;q^m)_\infty}{(q)_\infty}\\
&= \prod_{t=1}^\infty \frac{(1-q^{mt})}{(1-q^t)}.
\end{align*}
\end{proof}
The proof of Theorem \ref{evens} (2) relies on the same basic method as the above proof, using the two identities from Theorem 3.1 in place of those from Theorems 3.2 and 3.3.
\begin{proof} [Proof of Theorem \ref{evens} (2)]
By Theorem 3.1, we have that for $m = 2a + 2b - 1$
\begin{equation*}
R_{(2,-1)}(a,b;q) = \frac{(q^m;q^m)^a_\infty}{(q)^a_\infty}\cdot\prod_{i=1}^a \theta(q^i;q^m)\prod_{1\leq i<j\leq a} \theta(q^{j-i},q^{i+j};q^m).
\end{equation*}
Considering then the fraction below, we find that the denominator cancels out, and we are left with
\begin{equation*}
\frac{R_{(2,-1)}(\frac{m-1}{4}+1,\frac{m-1}{4}-1;q)}{R_{(2,-1)}(\frac{m-1}{4},\frac{m-1}{4};q)}=
\frac{(q^m;q^m)_\infty}{(q)_\infty} \theta(q^{\frac{m-1}{4}+1};q^m) \prod_{i=1}^\frac{m-1}{4} \theta(q^{\frac{m-1}{4}+1+i};q^m)\theta(q^{\frac{m-1}{4}+1-i};q^m).
\end{equation*} 
Reindexing, we can rewrite the right hand side of this equation as
\begin{align*}
\frac{(q^m;q^m)_\infty}{(q)_\infty} \theta(q^{\frac{m-1}{4}+1};q^m) &\cdot\prod_{i=1}^\frac{m-1}{4} \theta(q^i;q^m) \prod_{i=\frac{m-1}{4}  + 2}^\frac{m+1}{2}\theta(q^i;q^m)\\
&= \frac{(q^m;q^m)_\infty}{(q)_\infty} \prod_{i=1}^\frac{m+1}{2}\theta(q^i;q^m)\\
&=  \theta(q^\frac{m-1}{2};q^m).
\end{align*}
Then, from Theorem 3.1, we have that for $m = 2a + 2b -1$,
\begin{equation*}
R_{(1,-1)}(a,b;q) = \frac{(q^m;q^m)^b_\infty}{(q)^b_\infty}\cdot\prod_{i=1}^b \theta(q^{i+a};q^m)\prod_{1\leq i<j\leq b} \theta(q^{j-i},q^{i+j-1};q^m).
\end{equation*}
We then find that
\begin{align*}
R_{(1,-1)}\left(\frac{m-1}{2}-1,1;q\right) &= \frac{(q^m;q^m)_\infty}{(q)_\infty} \cdot \theta(q^{\frac{m-1}{2}};q^m).
\end{align*}
Then, we have that
\begin{align*}
\frac{R_{(1,-1)}(\frac{m-1}{2}-1, 1;q)\cdot R_{(2,-1)}(\frac{m-1}{4},\frac{m-1}{4};q)}{R_{(2,-1)}(\frac{m-1}{4}+1,\frac{m-1}{4}-1;q)} &= \frac{(q^m;q^m)_\infty}{(q)_\infty} \cdot \theta(q^{\frac{m-1}{2}};q^m) \cdot \frac{1}{\theta(q^\frac{m-1}{2};q^m)}\\
&= \frac{(q^m;q^m)_\infty}{(q)_\infty}\\
&= \prod_{t=1}^\infty \frac{(1-q^{mt})}{(1-q^t)}.
\end{align*}
\end{proof}

\section{Proof of Theorem \ref{aps}}\label{proof3}
\subsection{Modular Forms}
Using modular forms, we can prove the existence of an infinite number of congruences resembling those of Ramanujan and explicitly calculate several examples. We first recall that \textit{Dedekind's eta-function} is the infinite product
\begin{equation}
\eta(z):=q^{1/24}\prod_{t=1}^{\infty}{(1-q^t)}
\end{equation}
where $q:=e^{2\pi i z}$. Then, it is clear that
\begin{align*}
\eta^{m-1}(24z)&=q^{m-1}\prod_{t=1}^\infty(1-q^{24t})^{m-1}\\
&=\sum_{n=0}^\infty d_m(n)q^{24n+m-1}.
\end{align*}
If $N$ is a positive integer, then we define $\Gamma_0(N)$ as the congruence subgroup
\begin{equation*}
\Gamma_0(N) := \left\{\begin{pmatrix} a&b\\c&d\end{pmatrix}\in SL_2(\Z) : c\equiv 0 \pmod{N}\right\}.
\end{equation*}
Furthermore, we recall a theorem from \cite[p.18]{CBMS}, regarding such eta-products:

\begin{theorem}
If $f(z)=\prod_{\delta|N}\eta(\delta z)^{r_\delta}$ is of integer weight $k= \frac{1}{2}\sum_{\delta | N} r_\delta$ and  satisfies
\begin{align*}
\sum_{\delta|N}\delta r_\delta&\equiv0\pmod{24}
\intertext{and}
 \sum_{\delta|N}\frac{N}{\delta}r_\delta&\equiv0\pmod{24},
\intertext{then $f(z)$ satisfies} 
\hspace{0.5cm} f\left(\frac{az+b}{cz+d}\right)&=\chi(d)(cz+d)^kf(z)
\end{align*}
 for every $\begin{pmatrix}a&b\\c&d\end{pmatrix}\in \Gamma_0(N)$. Note that we define $\chi(d):=\left(\frac{(-1)^ks}{d}\right)$, where $s :=\prod_{\delta|N}\delta^{r_{\delta}}$. 
\end{theorem}

From Theorem 1.65 in \cite[p. 18]{CBMS}, we have the following result, which determines that these functions are cusp forms of weight $\frac{m-1}{2}$ with respect to $\Gamma_0(576)$.

\begin{proposition} \label{cusp}If $m \geq 2$, then $\eta^{m-1}(24z)\in S_{\frac{m-1}{2}}(\Gamma_0(576),\chi)$. In particular, this form has integer weight if and only if $m$ is odd.
\end{proposition}
\vspace{0.3cm}
If $f(z)$ is a modular form, then we can act on it with \textit{Hecke operators}, which we define as follows: if $f(z)=\sum_{n=0}^{\infty}a(n)q^n\in M_k(\Gamma_0(N),\chi)$, then the action of the \textit{Hecke operator} $T_{p,k,\chi}$ on $f(z)$ is defined by
\begin{equation*}
f(z)\hspace{0.15cm}|\hspace{0.15cm} T_{p,k,\chi}:=\sum_{n=0}^{\infty}(a(pn)+\chi(p)p^{k-1}a(n/p))q^n.
\end{equation*}
\cite[p. 21]{CBMS}.  We then recall the result in \cite[p. 43]{CBMS}:
\begin{theorem}\label{blackbox}
If $f(z)$ is a modular form of integer weight $k$, then there exist infinitely many primes $p$, such that
\begin{equation*}
\hspace{0.5cm} f(z)\hspace{0.15cm}|\hspace{0.15cm}T_{p,k,\chi}\equiv0\pmod{m}.
\end{equation*}
\end{theorem}
\boop
\subsection{Proof of Theorem \ref{aps}} Using the above machinery, we prove that expressions $\Psi_m=\sum d_m(n)q^n$ exhibit Ramanujan-like congruences in the coefficients $d_m(n)$.
\begin{proof}[Proof of Theorem \ref{aps}]
%STEP 1
Let $m = p^t$ for some prime $p$. First, we notice that
\begin{align*}
\sum_{n=0}^\infty d_m(n)q^n &= \prod_{t=1}^\infty\frac{1-q^{mt}}{1-q^t}\\
&\equiv\prod_{t=1}^{\infty}(1-q^t)^{m-1}\pmod{p}.
\end{align*} 
Then, we see that
\begin{align*}
\eta^{m-1}(24z)&=q^{m-1}\prod_{t=1}^\infty(1-q^{24t})^{m-1}\\
&= \sum_{n=0}^\infty d_m(n)q^{24n+m-1}.
\end{align*}
We then define $b_m(n)$ such that $b_m(24n+m-1) = d_m(n)$, so we have
\begin{equation*}
\eta^{m-1}(24z) = \sum_{n=0}^\infty b_m(n)q^n.
\end{equation*}
%STEP 2
We first consider the case where $p$ is an odd prime. If $p$ is odd, then we have that $m-1\in2\Z$, so our eta product has weight $k = {\frac{m-1}{2}} \in \Z$.
%STEP 3
Then, by Proposition \ref{cusp},
\begin{equation*}
f(z)=\eta^{m-1}(24z)\in S_{\frac{m-1}{2}}(\Gamma_0(N),\chi).
\end{equation*}
%STEP 4
Then, since $\eta^{m-1}(24z)$ has integer weight, we can apply Theorem \ref{blackbox}, which states that there exist infinitely many primes $l$ such that
\begin{align}
\label{infp}\sum_{n=0}^{\infty}(b_m(ln)+\chi(l)l^{(k-1)}b_m(n/l))q^n&\equiv0\pmod{p}.\\
\intertext{Then}
\notag\hspace{0.5cm} b_m(ln)+\chi(l)l^{(k-1)}b_m(n/l)&\equiv0\pmod{p}
\end{align}\\
for each $n$. Suppose now that $l \not|$ $n$. Then, $b_m(n/l) = 0$, so we have that
\begin{equation*}
b_m(l n) \equiv 0 \pmod{p}.
\end{equation*}
Therefore, since $b_m(24n+m-1) = d_m(n)$, we have 
\begin{equation*}
d_m\left(\frac{l n - m+1}{24}\right) \equiv 0 \pmod{p}
\end{equation*}
where $n\not\equiv 0 \pmod{l}$. Choosing an appropriate representative for the equivalence class of $n \in \Z/l\Z$, we find that 
\begin{equation*}
\frac{l n - m+1}{24}
\end{equation*}
is of the form $l^2n + d$ for some $d\in \Z$, so
\begin{equation*}
d_m(l^2n+d) \equiv 0 \pmod{p}.
\end{equation*}
Thus, since there exist infinitely many primes $l$ satisfying \eqref{infp}, we have that there exist infinitely many non-nested sequences $cn+d$ such that
\begin{equation*}
d_m(cn+d) \equiv 0 \pmod{p}.
\end{equation*}

We now consider the remaining case where $p=2$, noting that the argument above still applies once we make a simple substitution. Define $\theta(z)\in M_{\frac{1}{2}}(\Gamma_0(4),\chi)$ as
\begin{equation*}
\theta(z) := 1 + 2q + 2q^2 + ... \equiv 1 \pmod{2}.
\end{equation*}
Then we see that
\begin{equation*}
\eta^{m-1}(24z) \equiv \eta^{m-1}(24z)\cdot \theta(z) \pmod{2}.
\end{equation*}
Thus, the argument above holds when applied to the modular form $\eta^{m-1}(24z)\cdot \theta(z)$, which has integer weight $\frac{m}{2} = 2^{t-1}$, yet retains congruence modulo 2 to the original form $\eta^{m-1}(24z)$.
\end{proof} 

%%%%%%%%%%%%%%%%%%%%%%%%%%%%%%%%%%%
\bibliographystyle{amsplain}

\end{document}